\documentclass{svjour3}
\usepackage{amssymb, amsfonts, epsfig, color, amscd, setspace, lineno}

\title{Limit set intersection theorems for Kleinian groups and a conjecture of Susskind}
\author{James W. Anderson}
\institute{Mathematical Sciences\\University of Southampton\\Southampton SO17 1BJ\\England\\ \email{j.w.anderson@soton.ac.uk}}
\date{}

\begin{document}

\maketitle

%\begin{abstract}

%\keywords{First keyword \and Second keyword \and More}
% \PACS{PACS code1 \and PACS code2 \and more}
% \subclass{MSC code1 \and MSC code2 \and more}
%\end{abstract}

\section{Introduction, history and definitions}
\label{introduction}

We continue here the investigation of the relationship between the intersection of a pair of subgroups of a Kleinian group, and in particular the limit set of that intersection, and the intersection of the limit sets of the subgroups.   Of specific interest is the extent to which the intersection of the limit sets being non-empty implies that the intersection of the subgroups is non-trivial.   

The main purpose of this note is to consider the conjecture, due to Susskind, given below as Conjecture 0.  

\medskip
\noindent
{\bf Conjecture 0 (Susskind)} {\em Let $\Gamma$ be a non-elementary Kleinian group acting on ${\mathbb H}^n$ for some $n\ge 2$.  Let $\Phi$ and $\Theta$ be non-elementary  subgroups of $\Gamma$.  We then have that $\Lambda_c(\Phi)\cap\Lambda_c(\Theta)\subseteq  \Lambda(\Phi\cap\Theta)$.}
\medskip

We are unfortunately not able at this point to give a complete proof of Conjecture 0.  We are able to show that Conjecture 0 holds in a particular case, given below as Proposition \ref{conical and uniform conical}.    We then use a result of Bishop and Jones \cite{bishop jones} to conclude that in fact, Conjecture 0 holds most of the time.

\medskip
\noindent
{\bf Proposition \ref{conical and uniform conical}} {\em Let $\Gamma$ be a non-elementary purely loxodromic Kleinian group acting on ${\mathbb H}^n$ for some $n\ge 2$ and let $\Phi$ and $\Theta$ be non-trivial subgroups of $\Gamma$.  We then have that $\Lambda_c(\Phi)\cap \Lambda_c^u (\Theta)\subseteq \Lambda_c (\Phi\cap\Theta)$.}

\medskip
In addition, we provide examples to show that Conjecture 0 is as reasonably sharp as it can be; these are given below as Propositions \ref{fuchsian most general counterexample} and \ref{grid example}.

\medskip
\noindent
{\bf Proposition \ref{fuchsian most general counterexample}} {\em For every finitely generated purely loxodromic Fuchsian group $\Gamma$ with $\Lambda(\Gamma) ={\mathbb S}^1_\infty$, there exist non-elementary subgroups $\Phi$ and $\Theta$ of $\Gamma$ so that $\Phi\cap \Theta$ is trivial and $\Lambda(\Phi)\cap \Lambda(\Theta)$ contains exactly two points.  In particular, we have that $\Lambda(\Phi\cap \Theta)\ne \Lambda(\Phi)\cap \Lambda(\Theta)$ for this $\Phi$ and $\Theta$.}

\medskip
\noindent
{\bf Proposition \ref{grid example}} {\em There exists a non-elementary purely loxodromic Fuchsian group $\Gamma$ and non-trivial subgroups $\Phi$ and $\Theta$ of $\Gamma$ so that $\Lambda(\Phi\cap\Theta) \setminus \Lambda_c(\Phi)\cap\Lambda_c(\Theta)$ is non-empty.}

\medskip
In light of Proposition \ref{grid example}, we cannot expect that the inclusion at the conclusion of Conjecture 0 is in fact an equality in general. 

\medskip
We begin with some definitions. A {\em Kleinian group} $\Gamma$ is a discrete subgroup of the group ${\rm Isom}^+({\mathbb H}^n)$ of the orientation-preserving isometries of the real hyperbolic space ${\mathbb H}^n$ for some $n\ge 2$.  (In the case that $n=2$, we will follow historical convention and refer to $\Gamma$ as a {\em Fuchsian group}.)  There are a number of references for basic facts about Kleinian groups, including Maskit \cite{maskit book}, Kapovich \cite{kapovich}, Matsuzaki and Taniguchi \cite{matsuzaki taniguchi}, and Ratcliffe \cite{ratcliffe}.

For a loxodromic element $\gamma$ of ${\rm Isom}^+({\mathbb H}^n)$, we let ${\rm fix}(\gamma)$ denote the pair of {\em fixed points} for the action of $\gamma$ on ${\mathbb S}^{n-1}_\infty$ and ${\rm axis}(\gamma)$  the {\em axis} of $\gamma$, which is the hyperbolic line in ${\mathbb H}^n$ whose endpoints at infinity in ${\mathbb S}^{n-1}_\infty$ are the two points of ${\rm fix}(\gamma)$ and which consequently is invariant under the action of $\langle \gamma\rangle$.   

For ease of exposition, we restrict our attention exclusively to {\em purely loxodromic} Kleinian groups, so that we are excluding those Kleinian groups that contain parabolic or elliptic elements, and we tailor our definitions and the statements of results accordingly.    We comment here that many of the extant results referenced in this note apply to general Kleinian groups, and we refer the interested reader to the references given for the most general statements.  Where relevant, we will remark on when the results proved in this note extend to general Kleinian groups.

The action of a Kleinian group $\Gamma$ on the sphere at infinity ${\mathbb S}^{n-1}_\infty$ of ${\mathbb H}^n$ partitions ${\mathbb S}^{n-1}_\infty$ into the {\em domain of discontinuity} $\Omega(\Gamma)$, which is the largest (possibly empty) open subset of ${\mathbb S}^{n-1}_\infty$ on which $\Gamma$ acts properly discontinuously, and the {\em limit set} $\Lambda(\Gamma)$, which is the closure in ${\mathbb S}^{n-1}_\infty$ of the set of fixed points of loxodromic elements of $\Gamma$.  

A Kleinian group is {\em non-elementary} if it is not virtually abelian, and is {\em elementary} otherwise.  The limit set of a non-elementary Kleinian group is a perfect set, and hence is uncountable.    A non-trivial purely loxodromic Kleinian group is elementary if and only if it is loxodromic cyclic.

For a non-elementary Kleinian group $\Gamma$ acting on ${\mathbb H}^3$, we have the following additional structure on its domain of discontinuity.    Assuming that its domain of discontinuity $\Omega(\Gamma)\subset {\mathbb S}^2_\infty$ is non-empty, each (connected) component $\Delta$ of $\Omega(\Gamma)$ then has the open unit disc ${\bf D}\subset {\bf C}$ as its universal cover.  The action of the {\em component stabilizer} $\Gamma_\Delta = \{ \gamma\in\Gamma \: |\: \gamma(\Delta) = \Delta\}$ of $\Delta$ in $\Gamma$ lifts to the action of a group of conformal homeomorphisms of ${\bf D}$, which are isometries of the natural hyperbolic metric on ${\bf D}$.  Hence, we see that $\Delta$ admits a canonical hyperbolic metric for which $\Gamma_\Delta$ acts by isometries, and so all of $\Omega(\Gamma)$ admits a canonical hyperbolic metric for which $\Gamma$ acts by isometries.  We refer to this as the {\em Poincar\'e metric} on $\Omega(\Gamma)$.   

However, this is specific for Kleinian groups acting on ${\mathbb H}^3$, and does not hold for general $n\ge 4$.  In higher dimensions, so for a Kleinian group $\Gamma$ acting on ${\mathbb H}^n$ for some $n\ge 4$, there is a natural conformal structure on the domain of discontinuity $\Omega(\Gamma)\subset {\mathbb S}^{n-1}_\infty$, but unfortunately this conformal structure does not play the same role as the Poincar\'e metric.  For information on this, we refer the interested reader to Apanasov \cite{apanasov}.  

A non-elementary Kleinian group $\Gamma$ acting on ${\mathbb H}^3$ with non-empty domain of discontinuity $\Omega(\Gamma)$ is {\em analytically finite} if the quotient surface $\Omega(\Gamma)/\Gamma$ has finite area in the Poincar\'e metric.  It is a fundamental theorem of Ahlfors that all finitely generated Kleinian groups acting on ${\mathbb H}^3$ with non-empty domain of discontinuity are necessarily analytically finite.  The converse holds for Fuchsian groups, but there are examples of Kleinian groups acting on ${\mathbb H}^3$ which show that the converse is false in general, in that there are infinitely generated Kleinian groups acting on ${\mathbb H}^3$ for which the quotient surface $\Omega(\Gamma)/\Gamma$ has finite area in its Poincar\'e metric.   We refer the reader to Kapovich \cite{kapovich} for a discussion of the failure of variants of the Ahlfors finiteness theorem to Kleinian groups acting on ${\mathbb H}^n$ for some $n\ge 4$. 

The definition we give above of the limit set is one of several characterizations of the limit set, all of which are equivalent for non-elementary Kleinian groups acting on ${\mathbb H}^n$ for some $n\ge 2$; we use it as our definition as it is the most useful of these equivalent characterizations for the approach we take herein.

The limit set has several useful properties that follow immediately from the definition, which we bring together in the following Lemma for ease of reference.

\begin{lemma}  Let $\Gamma$ be a non-elementary Kleinian group acting on ${\mathbb H}^n$ for some $n\ge 2$.  We then have the following:
\begin{enumerate}
\item If $\Phi\subseteq\Gamma$, then $\Lambda(\Phi)\subseteq\Lambda(\Gamma)$; 
\item If $\Phi$ is a finite index subgroup of $\Gamma$, then $\Lambda(\Phi) =\Lambda(\Gamma)$;
\item If $\Phi$ is a non-trivial normal subgroup of $\Gamma$, then $\Lambda(\Phi) = \Lambda(\Gamma)$.
\end{enumerate}
\label{limit set properties}
\end{lemma}

There are two classes of limit points that are of particular interest to us in this note.  The limit point $x\in \Lambda(\Gamma)$ of the Kleinian group $\Gamma$ is a {\em conical limit point}, or alternatively a {\em point of approximation} or a {\em radial limit point}, if for every ray $r\subset {\mathbb H}^n$ with $x$ as its endpoint at infinity, there exists a compact set $K\subset {\mathbb H}^n$ and a sequence $\{ \gamma_k\}$ of distinct elements of $\Gamma$ so that $\gamma_k(K)\cap r$ is non-empty for all $k$.   However, the intersections $\gamma_k(K)\cap r$ may be very unevenly spaced along $r$.  Denote the set of conical limit points of $\Gamma$ by $\Lambda_c(\Gamma)$.

The limit point $x\in \Lambda(\Gamma)$ is a {\em uniform conical limit point} if for every ray $r\subset {\mathbb H}^n$ with $x$ as its endpoint at infinity, there exists a compact set $K\subset {\mathbb H}^n$ and a sequence $\{ \gamma_k\}$ of distinct elements of $\Gamma$ so that $r\subset \cup_k \gamma_k(K)$.   In particular, the intersections $ \gamma_k(K)\cap r$ are extremely evenly spaced along $r$.  Denote the set of uniform conical limit points by $\Lambda_c^u (\Gamma)$.  

It follows immediately from the definitions that $\Lambda_c^u (\Gamma)\subset \Lambda_c (\Gamma)$.  Also, the fixed points of a loxodromic element $\gamma$ of $\Gamma$ are necessarily contained in $\Lambda_c^u (\Gamma)$, as we may take $r$ to be the ray contained in the hyperbolic line ${\rm axis}(\gamma)$ and the compact set $K$ to be the closed hyperbolic line segment in $r$ between any point of $r$ and its translate by $\gamma^{\pm 1}$. 

If we consider the projection of any such ray $r$ to the quotient $N_\Gamma ={\mathbb H}^n/\Gamma$, then for a conical limit point, there exists a compact set $K$ in $N_\Gamma$ and a sequence $\{ x_k\}$ of points exiting $r$ so that the projection of the $x_k$ to $N_\Gamma$ all lie in $K$, whereas for a uniform conical limit point, the projection of all of $r$ to $N_\Gamma$ lies in a compact subset of $N_\Gamma$. 

A Kleinian group $\Gamma$ is {\em convex co-compact} if $\Lambda(\Gamma) = \Lambda_c (\Gamma)$.   In all dimensions $n\ge 2$, a convex co-compact Kleinian group acting on ${\mathbb H}^n$ is necessarily finitely generated, but the converse only holds for Fuchsian groups.   In dimension $n\ge 3$, there exist purely loxodromic finitely generated Kleinian groups acting on ${\mathbb H}^n$ which are not convex co-compact.  

\medskip

The question of the relationship between $\Lambda(\Phi)\cap\Lambda(\Theta)$ and $\Lambda(\Phi\cap\Theta)$ for subgroups $\Phi$ and $\Theta$ of a Kleinian group has been considered by a number of different authors in a number of different contexts.  We give here a short survey of known results.   We remind the reader that we have adapted the statements to the case of purely loxodromic Kleinian groups, and we refer the interested reader to the original papers for the complete statements. 

Chronologically, the first result of this sort is due to Maskit.

\begin{theorem} [Theorem 3 in Maskit \cite{maskit intersection}] Let $\Gamma$ be a non-elementary purely loxodromic Kleinian group acting on ${\mathbb H}^3$ with non-empty domain of discontinuity $\Omega(\Gamma)$.  Let $\Delta_1$ and $\Delta_2$ be components of $\Omega(\Gamma)$ with corresponding component stabilizers  $\Gamma_k = \{ \gamma\in\Gamma\: |\: \gamma(\Delta_k) =\Delta_k\}$, and assume that both $\Delta_k/\Gamma_k$ have finite area.  We then have that 
\[ \Lambda(\Gamma_1\cap\Gamma_2) = \Lambda(\Gamma_1)\cap\Lambda(\Gamma_2) = \overline{\Delta_1}\cap\overline{\Delta_2}. \]
\label{maskit intersection}
\end{theorem}

We note that by the Ahlfors finiteness theorem, the hypothesis that both $\Delta_k/\Gamma_k$ have finite area is always satisfied for finitely generated $\Gamma$, and (as has already been noted) there are examples of infinitely generated $\Gamma$ for which this hypothesis is also satisfied.  

In the same paper, Maskit also shows (see Theorem 4 of \cite{maskit intersection}), using similar techniques to those used in the proof of Theorem \ref{maskit intersection} (and with the same hypotheses), that if $\gamma\in\Gamma$ is a loxodromic element for which ${\rm fix}(\gamma)\cap \Lambda(\Gamma_k)$ is non-empty, then there exists an integer $n_\gamma >0$ so that $\gamma^{n_\gamma}\in \Gamma_k$.  That is, since 
\[ {\rm fix}(\gamma)\Lambda(\Gamma_k) = \Lambda(\langle \gamma\rangle)\cap \Lambda(\Gamma_k) = \Lambda(\langle \gamma\rangle \cap\Gamma_k) \]
is non-empty, we have that $\langle\gamma\rangle\cap\Gamma_k$ is non-trivial.

The next several results we state also hold specifically for Kleinian groups acting on ${\mathbb H}^3$. 

\begin{theorem} [Theorems 4.1 and 5.2 of Anderson \cite{anderson analytically geometrically}] Let $\Gamma$ be a non-elementary purely loxodromic Kleinian group acting on ${\mathbb H}^3$ with non-empty domain of discontinuity $\Omega(\Gamma)$.  Let $\Phi$ and $\Theta$ be non-trivial subgroups of $\Gamma$, where $\Phi$ is convex co-compact (possibly loxodromic cyclic) and $\Theta$ is non-elementary and analytically finite.  We then have that $\Lambda(\Phi\cap\Theta) = \Lambda(\Phi)\cap\Lambda(\Theta)$. 
\label{anderson analytically geometrically}
\end{theorem}

As above, note that in the case that $\Phi =\langle \varphi\rangle$ is loxodromic cyclic, the conclusion of Theorem \ref{anderson analytically geometrically} can be rephrased as saying that there exists an integer $n_\varphi >0$ so that $\varphi^{n_\varphi}\in \Theta$.  

Making use of the resolution of Marden's Tameness Conjecture by Agol \cite{agol} and Calegari and Gabai \cite{calegari gabai}, which states that all finitely generated Kleinian groups acting on ${\mathbb H}^3$ are topologically tame, we have the following (which subsumes the main results of Anderson \cite{anderson lsit}).  (See also Soma \cite{soma}.)  Recall that a Kleinian group $\Gamma$ acting on ${\mathbb H}^3$ is {\em virtually fibered} over a finitely generated subgroup $\Phi$ if there exists a finite index subgroup $\Gamma^0$ so that ${\mathbb H}^3/\Gamma^0$ is a closed hyperbolic $3$-manifold fibering over the circle ${\mathbb S}^1$ with fiber subgroup $\Gamma^0\cap\Phi$. 

\begin{theorem} [Theorems A and C of Anderson \cite{anderson topologically}] Let  $\Gamma$ be a non-elementary purely loxodromic Kleinian group acting on ${\mathbb H}^3$, and let $\Phi$ and $\Theta$ be non-trivial finitely generated subgroups of $\Gamma$, where $\Theta$ is non-elementary.  We then have that either $\Lambda(\Phi\cap\Theta) = \Lambda(\Phi)\cap\Lambda(\Theta)$ or we are in the exceptional case that $\Gamma$ is virtually fibered over either $\Phi$ or $\Theta$.
\label{anderson topologically}
\end{theorem}

Note that in the case that $\Phi =\langle \varphi\rangle$ is loxodromic cyclic, the conclusion of Theorem \ref{anderson topologically} in the case that $\Gamma$ is not virtually fibered over $\Theta$ can be rephrased as saying that  there exists an integer $n_\varphi >0$ so that $\varphi^{n_\varphi}\in \Theta$.   In the case that $\Gamma$ is virtually fibered over $\Theta$, then we necessarily have a non-trivial element $\gamma\in\Gamma$ so that ${\rm fix}(\gamma)\subset \Lambda(\Theta) = {\mathbb S}^2_\infty$ but for which no proper power of $\gamma$ lies in $\Theta$.

In particular, in light of the above results, we have a comprehensive understanding of the behavior of the limit set $\Lambda(\Phi\cap\Theta)$ of the intersection of two subgroups $\Phi$ and $\Theta$ of a Kleinian group $\Gamma$ acting on ${\mathbb H}^3$, in that $\Lambda(\Phi\cap\Theta) = \Lambda(\Phi)\cap\Lambda(\Theta)$ for finitely generated $\Phi$ and $\Theta$, as well as for some infinitely generated through analytically finite subgroups as well.    However, nothing is known about the behavior of $\Lambda(\Phi\cap\Theta)$ in terms of $\Lambda(\Phi\cap\Theta)$ for general infinitely generated $\Phi$ and $\Theta$. 

We have the following generalization to Kleinian groups acting on ${\mathbb H}^n$ for general $n\ge 2$, which generalizes the main result of Susskind's thesis, where he proved this result in the case $n=3$.  As we have phrased it here, Theorem \ref{susskind swarup} does some overlap some of the previously stated results in the cases $n=2$ and $n=3$.

\begin{theorem} [Corollaries 1 and 3 of Susskind and Swarup \cite{susskind swarup}] Let $\Gamma$ be a non-elementary purely loxodromic Kleinian group acting on ${\mathbb H}^n$ for some $n\ge 2$, and let $\Phi$ and $\Theta$ be non-trivial, convex co-compact subgroups of $\Gamma$ with $\Theta$ non-elementary.  We then have that $\Lambda(\Phi)\cap\Lambda(\Theta) = \Lambda(\Phi\cap\Theta)$. 
\label{susskind swarup}
\end{theorem}

Yet again, we note that in the case that $\Phi =\langle \varphi\rangle$ is loxodromic cyclic, the conclusion of Theorem \ref{susskind swarup} can be rephrased as saying that there exists an integer $n_\varphi >0$ so that $\varphi^{n_\varphi}\in \Theta$.  

Though it is an interesting question in its own right, we do not consider in this note the question of determining the most general context in which limit set intersection theorems of the above sort hold.  We do note that the basic requirements are that the groups in question act reasonably on a space, where the action of the group leaves a physical trace on the space (akin to the limit set in the case of Kleinian groups).  Among these cases which have been considered are hyperbolic groups and relatively hyperbolic groups acting on compact metric spaces.  We refer the interested reader to Yang \cite{yang} for the case of relatively quasiconvex subgroups of a relatively hyperbolic group, as well as an introduction to what is known for hyperbolic groups.

\section{The Susskind conjecture}
\label{section susskind conjecture}

We know from Lemma \ref{limit set properties} that for subgroups $\Phi$ and $\Theta$ of a Kleinian group $\Gamma$, the inclusion $\Lambda(\Phi\cap\Theta)\subseteq \Lambda(\Phi)\cap\Lambda(\Theta)$ always holds.   The question of interest is, how much larger can $\Lambda(\Phi)\cap\Lambda(\Theta)$ be than $\Lambda(\Phi\cap\Theta)$.   The main conjecture that guides and shapes our discussion is the following, due to Susskind (oral communication).

\medskip
\noindent
{\bf Conjecture 0 (Susskind)} {\em Let $\Gamma$ be a non-elementary Kleinian group acting on ${\mathbb H}^n$ for some $n\ge 2$.  Let $\Phi$ and $\Theta$ be non-elementary  subgroups of $\Gamma$.  We then have that $\Lambda_c(\Phi)\cap\Lambda_c(\Theta)\subseteq  \Lambda(\Phi\cap\Theta)$.}
\medskip

In Section \ref{two examples}, we present two examples to show that the statement of Conjecture 0 cannot be significantly strengthened.   Looking ahead, one example shows that we cannot expect to have $\Lambda_c(\Phi)\cap\Lambda_c(\Theta)\subseteq  \Lambda_c (\Phi\cap\Theta)$ for general $\Phi$ and $\Theta$, and the other example shows that we cannot expect to have $\Lambda(\Phi)\cap\Lambda(\Theta)\subseteq  \Lambda(\Phi\cap\Theta)$ for general $\Phi$ and $\Theta$.

Unfortunately, we are not yet able to provide a complete proof of Conjecture 0, though we are able to come close in a measure-theoretic sense (as described below).  Specifically, we prove the following Proposition and consider its implications for Conjecture 0.

\begin{proposition} Let $\Gamma$ be a non-elementary purely loxodromic Kleinian group acting on ${\mathbb H}^n$ for some $n\ge 2$ and let $\Phi$ and $\Theta$ be non-trivial subgroups of $\Gamma$.  We then have that $\Lambda_c(\Phi)\cap \Lambda_c^u (\Theta)\subseteq \Lambda_c (\Phi\cap\Theta)$. 
\label{conical and uniform conical}
\end{proposition}

\begin{proof} Let $x\in \Lambda_c(\Phi)\cap \Lambda_c^u (\Theta)$ be any point and let $r\subset {\mathbb H}^n$ be any ray with $x$ as its endpoint at infinity.  By the definition of $\Lambda_c (\Phi)$ given above, there exists a compact set $K'\subset {\mathbb H}^n$ and a sequence $\{ \varphi_k\}$ of distinct elements of $\Phi$ so that $\varphi_k(K')\cap r$ is non-empty for all $k$.  By the definition of $\Lambda_c^u (\Theta)$ given above, there exists a compact set $K''\subset {\mathbb H}^n$ and a sequence $\{ \theta_\ell\}$ of distinct elements of $\Theta$ so that $r\subset \cup_\ell \theta_\ell(K'')$.  

Project the translates $\varphi_k(K')$ of $K'$ by the elements of the sequence $\{ \varphi_k\}\subset \Phi$ to the quotient ${\mathbb H}^n/\Theta$.   Since each $\varphi_k(K')\cap r$ is non-empty and since $r$ projects to a compact subset of ${\mathbb H}^n/\Theta$, we see that the $\varphi_k(K')$ project to a compact subset of ${\mathbb H}^n/\Theta$.  If we wish to be explicit, let $\varepsilon = {\rm diam}(K')$ be the hyperbolic diameter of $K'$ in ${\mathbb H}^n$, and note that the $\varphi_k(K')$ project to the closure of the $\varepsilon$-neighborhood of the projection of $K''$ to ${\mathbb H}^n/\Theta$.

Choose a point $z_0\in K'$.  Diagonalizing, we have that the sequence $\{ \theta_m  \varphi_m (z_0)\}$ is contained in the just-described compact subset of ${\mathbb H}^n/\Theta$ and so has a convergent subsequence, which we also denote by $\{ \theta_m  \varphi_m (z_0)\}$.  The discreteness of $\Gamma$ yields that $\theta_m \varphi_m = {\rm id}_\Gamma$ for infinitely many $m$, where ${\rm id}_\Gamma$ is the identity element of $\Gamma$. 

In particular, by yet again passing to a subsequence (and reindexing, and starting the indexing in the resulting sequence at $m=0$), we have that $\theta_m \varphi_m = \theta_0 \varphi_0 = {\rm id}_\Gamma$ for all $m\ge 0$.  This immediately yields that $\Phi\cap\Theta$ is non-trivial, as $\theta_0^{-1}\theta_m =\varphi_0 \varphi_m^{-1}\in\Phi\cap\Theta$ for all $m\ge 0$.  (Recall that we assumed initially that the elements of the sequences $\{ \varphi_k\}$ and  $\{ \theta_\ell\}$ were distinct, and so none of the $\theta_0^{-1}\theta_m =\varphi_0 \varphi_m^{-1}$ are the identity.)

To see that $x\in \Lambda_c (\Phi\cap\Lambda)$, we merely need to note that taking the sequence $\{ (\varphi_0 \varphi_m^{-1})^{-1} = \varphi_m \varphi_0^{-1}\}$ of distinct elements of $\Phi\cap\Theta$ and the compact set $\varphi_0 (K')\subset {\mathbb H}^n$, we have that $\varphi_m \varphi_0^{-1} (\varphi_0(K'))\cap r = \varphi_m (K')\cap r$ is non-empty for all $m$.   {\hfill ${\bf QED}$ \\ \vspace{.15cm}}
\end{proof}

It follows from work of Bishop and Jones \cite{bishop jones} that in a quantifiable sense, most points of $\Lambda_c(\Gamma)$ are in fact points of $\Lambda_c^u (\Gamma)$.   To make this precise, we need the following definition.  Given a non-elementary Kleinian group $\Gamma$ acting on ${\mathbb H}^n$ for some $n\ge 2$ and a point $z_0\in {\mathbb H}^n$, consider the {\em Poincar\'e series} 
\[ \sum_{\gamma\in\Gamma} \exp(-s \: {\rm d}_{{\mathbb H}^n} (z_0, \gamma(z_0))) \]
for $s >0$, where ${\rm d}_{{\mathbb H}^3} (\cdot, \cdot)$ is the hyperbolic distance in ${\mathbb H}^3$.  When $s$ is large, this series converges, and when $s$ is sufficiently close to zero, this series diverges.  In particular, there is a value $\delta(G) >0$, the {\em critical exponent}, of the parameter $s$ so that the series converges for $s >\delta(G)$ and diverges for $s < \delta(G)$.  

The critical exponent is closely connected to the {\em Hausdorff dimension} ${\rm dim}(\Lambda(\Gamma))$ of the limit set, and in fact in the same paper, Bishop and Jones provide a concise and complete proof of the fact that for a non-elementary Kleinian group $\Gamma$ acting on ${\mathbb H}^n$ for some $n\ge 2$, the equality $\delta(G) = {\rm dim}(\Lambda_c (\Gamma))$ holds.

The relationship between $\Lambda_c (\Gamma)$ and $\Lambda_c^u (\Gamma)$ is given in the following result.  In particular, the last sentence of Theorem \ref{bishop jones most are uniform} can be rephrased as saying that ${\rm dim}(\Lambda_c(\Gamma)) = {\rm dim}(\Lambda_c^u (\Gamma))$.

\begin{theorem} [Corollary 2.5 of Bishop and Jones \cite{bishop jones}] If $G$ is any non-elementary, discrete M\"obius group, $\varepsilon >0$ and $x\in M={\mathbb H}^n/G$, then there is an $R=R(\varepsilon,x)$ such that the set of directions (i.e., unit tangents at $x$) which correspond to geodesic rays starting at $x$ which never leave the ball of radius $R$ around $x$ has dimension $\ge \delta(G)-\varepsilon = {\rm dim}(\Lambda_c (G))-\varepsilon$.  In particular, the set of directions at $x$ of bounded geodesic rays has dimension exactly $\delta(G)$.
\label{bishop jones most are uniform}
\end{theorem}

The proof of the following Corollary follows immediately from the observation that for a loxodromic cyclic Kleinian group $\Phi = \langle \gamma\rangle$ acting on $\mathbb{H}^n$ for any $n\ge 2$, we have that both of the fixed points of $\gamma$ are points of $\Lambda_c^u (\Phi)$.  

\begin{corollary} Let $\Gamma$ be a purely loxodromic Kleinian group acting on ${\mathbb H}^n$ for some $n\ge 2$ and let $\Phi$ be a non-trivial subgroup of $\Gamma$.  Let $x\in \Lambda_c(\Phi)$ and let $\gamma$ be a non-trivial element of $\Gamma$ fixing $x$.  Then there exists $m\ge 1$ for which $\gamma^m\in \Phi$. 
\label{cyclic corollary}
\end{corollary}

One common aspect of the proofs of many of the results given in the previous Section is a  preliminary result of the form: Let $\Gamma$ be a Kleinian group and let $\Phi$ and $\Theta$ be subgroups of $\Gamma$.  Suppose that $x\in \Lambda_c (\Phi)\cap\Lambda(\Theta)$ and that $\Theta$ satisfies a finiteness condition.  Then $x\in \Lambda(\Phi\cap\Theta)$.    

The finiteness conditions used above are that $\Gamma$ acts on ${\mathbb H}^3$ and $\Theta$ is analytically finite (Theorem \ref{anderson analytically geometrically}); that $\Gamma$ acts on ${\mathbb H}^3$ and $\Theta$ is finitely generated (Theorem \ref{anderson topologically}), or that $\Theta$ is convex co-compact for general $n\ge 2$ (Theoerem \ref{susskind swarup}).    Proposition \ref{conical and uniform conical} shows that we can often replace the finiteness condition on $\Theta$ with the condition that $x\in \Lambda_c^u (\Theta)$.

\medskip

There exists one situation in which we can prove the strongest possible result, that $\Lambda(\Phi)\cap\Lambda(\Theta) = \Lambda(\Phi\cap\Theta)$, in the case that $\Phi$ and $\Theta$ are non-trivial normal subgroups of $\Gamma$. 

\begin{lemma} Let $\Gamma$ be a non-elementary purely loxodromic Kleinian group acting on ${\mathbb H}^n$ for some $n\ge 2$.  Let $\Phi$ and $\Theta$ be non-trivial normal subgroups of $\Gamma$.  We then have that $\Phi\cap\Theta$ is non-empty and hence that $\Lambda(\Phi)\cap \Lambda(\Theta) = \Lambda(\Phi\cap\Theta)$.
\label{normal intersection}
\end{lemma}

\begin{proof} Since $\Phi$ and $\Theta$ are non-elementary, there exist loxodromic elements $\varphi\in\Phi$ and $\theta\in\Theta$.  Since $\Theta$ is normal in $\Gamma$, there exists some $\theta_0\in\Theta$ so that $\varphi\theta\varphi^{-1} = \theta_0$, from which we see that $\varphi\theta = \theta_0\varphi$; similarly, since $\Phi$ is normal in $\Gamma$, there exists some $\varphi_0\in\Phi$ so that $\theta^{-1}\varphi\theta = \varphi_0$, from which we see that $\varphi\theta = \theta\varphi_0$.  

Therefore, we have that $\theta_0\varphi= \theta\varphi_0$, and hence that $\theta^{-1}\theta_0 = \varphi_0\varphi^{-1}$.  If $\theta^{-1}\theta_0$ is non-trivial, then we see immediately that $\Phi\cap\Theta$ is non-empty, as it contains the non-trivial element $\theta^{-1}\theta_0 = \varphi_0\varphi^{-1}$ of $\Gamma$. 

If $\theta^{-1}\theta_0$ is trivial, then we have that $\varphi\theta = \theta\varphi$.  Since $\theta$ and $\varphi$ are loxodromic, this implies that there exists some (primitive) loxodromic element $\gamma\in \Gamma$ and non-zero integers $n_\theta$ and $n_\varphi$ so that $\theta = \gamma^{n_\theta}$ and $\varphi = \gamma^{n_\varphi}$.  Again, we have that $\Phi\cap\Theta$ is non-trivial, as this intersection then contains $\gamma^{n_\varphi n_\theta}$.  

In either case, the intersection $\Phi\cap\Theta$ is then a non-trivial normal subgroup of both $\Phi$ and $\Theta$, and consequently we have by Lemma \ref{limit set properties} that $\Lambda(\Phi\cap\Theta) =\Lambda(\Phi) =\Lambda(\Theta)$.  {\hfill ${\bf QED}$ \\ \vspace{.15cm}}
\end{proof}

We close this noting that the results in this Section hold in full generality, for all Kleinian groups in all dimensions $n\ge 2$.   

\section{Two examples}
\label{two examples}

In this Section, we present two examples to demonstrate the sharpness of Conjecture 0, as discussed in the previous Section.  We begin with an example to show that we cannot expect to have the strongest result we might hope for, that $\Lambda(\Phi\cap\Theta) = \Lambda(\Phi)\cap \Lambda(\Theta)$, for non-elementary subgroups $\Phi$ and $\Theta$ of even a Fuchsian group $\Gamma$. 

\begin{proposition} For every finitely generated purely loxodromic Fuchsian group $\Gamma$ with $\Lambda(\Gamma) ={\mathbb S}^1_\infty$, there exist non-elementary subgroups $\Phi$ and $\Theta$ of $\Gamma$ so that $\Phi\cap \Theta$ is trivial and $\Lambda(\Phi)\cap \Lambda(\Theta)$ contains exactly two points.  In particular, we have that $\Lambda(\Phi\cap \Theta)\ne \Lambda(\Phi)\cap \Lambda(\Theta)$ for this $\Phi$ and $\Theta$. 
\label{fuchsian most general counterexample}
\end{proposition}

\begin{proof} We start by noting that the quotient surface $S ={\mathbb H}^2/\Gamma$ has positive genus, and so there exists a simple, closed, non-separating geodesic $a$ on $S$.  Let $\gamma$ be a (necessarily primitive) loxodromic element of $\Gamma$ representing $a$, so that $\ell = {\rm axis}(\gamma)$ projects to $a$; we note that $\gamma$ is determined only up to conjugacy and taking inverses in $\Gamma$.  

Let $\Gamma^0$ be a non-trivial normal subgroup of $\Gamma$ which contains no non-trivial power of $\gamma$, or phrased alternatively, so that $\gamma^n\not\in\Gamma^0$ for all $n\ge 1$.  Such a subgroup always exists; for instance, we can choose $\Gamma^0$ to be the commutator subgroup $\Gamma^0 = [\Gamma, \Gamma]$ of $\Gamma$ since $a$ was assumed to be non-separating.  

We make the following observations.  First, we have that $\ell\cup {\rm fix}(\gamma)$ is precisely invariant under the identity in $\Gamma^0$, by which we mean that $\mu(\ell\cup {\rm fix}(\gamma))$ is disjoint from $\ell\cup {\rm fix}(\gamma)$ for all non-trivial elements $\mu\in\Gamma^0$.   The precise invariance under the identity of $\ell$ follows immediately from the assumption that $a$ is a simple geodesic and the fact that no non-trivial power of $\gamma$ lies in $\Gamma^0$.  The fact that $\mu({\rm fix}(\gamma))$ is disjoint from ${\rm fix}(\gamma)$ follows from the fact that no non-trivial power of $\gamma$ lies in $\Gamma^0$, together with the discreteness of $\Gamma$.

The second is that for any sequence $\{ \mu_n \}$ of distinct elements of $\Gamma^0$, the spherical diameters (in the closed disc ${\mathbb H}^2\cup {\mathbb S}^1_\infty$) of the translates $\{ \mu_n(\ell)\}$ of $\ell$ go to $0$, and so in particular, any compact set in ${\mathbb H}^2$ can intersect only finitely many translates of $\ell$ by elements of $\Gamma^0$.  This follows from the fact that since $\ell$ projects to a closed geodesic on $S$, standard covering theory implies that the translates of $\ell$ under $\Gamma$ cannot accumulate to any point of ${\mathbb H}^2$. 

We are now able to mimic some of the ideas contained in Maskit's proof of the combination theorems for Kleinian groups.  (See Maskit \cite{maskit book}.)  As these ideas may be unfamiliar to some readers, we give a sketch of the proof here.

Consider the complement ${\mathbb H}^2 \setminus \Gamma^0 (\ell)$ of the translates of $\ell$ by elements of $\Gamma^0$ in ${\mathbb H}^2$.   The two observations given just above together imply that there exist two components $X_0$ and $X_1$ of ${\mathbb H}^2 \setminus \Gamma^0(\ell)$ which contain $\ell$ in their boundary.   Moreover, since we are working with the complement of translates of the single line $\ell$, every component of ${\mathbb H}^2 \setminus \Gamma^0(\ell)$ is equivalent under the action of $\Gamma^0$ to either $X_0$ or $X_1$.  (It may be that $X_0$ and $X_1$ are themselves equivalent under the action of $\Gamma^0$.)   Let $\Phi_k = {\rm stab}_{\Gamma^0} (X_k)$ be the stabilizer of $X_k$ in $\Gamma^0$. 

As we have made no assumptions about the relationship between $\gamma$ and $\Gamma^0$ beyond the ones made above, there are two cases.  One is that $X_0$ and $X_1$ are not equivalent under the action of $\Gamma^0$ and the other is that there exists $\varphi_0\in \Gamma^0$ satsifying $\varphi_0 (X_0) = X_1$.  

In the former case, we have that every translate of $\ell$ on the boundary of $X_k$ is the translate of $\ell$ by an element of $\Phi_k$.  (Otherwise, there exists an element $\mu$ of $\Gamma^0$ so that $\mu(\ell)$ lies in the boundary of $X_k$ but for which $\mu(X_k)\ne X_k$, and we then immediately see that $\mu(X_{1-k}) = X_k$, contrary to our assumption.) We can then argue that $\Gamma^0 = \langle \Phi_0, \Phi_1\rangle =   \Phi_0\ast \Phi_1$ by taking any element $\mu\in\Gamma^0$ and arguing by induction on the number of translates of $\ell$ separating $\ell$ from $\mu(\ell)$.   The fact that $ \langle \Phi_0, \Phi_1\rangle =   \Phi_0\ast \Phi_1$ follows from the precise invariance of $\ell$ under the identity in $\Gamma^0$, together with the standard ping-pong argument.

Label the endpoints at infinity of $\ell$ as $x$ and $y$, so that ${\rm fix}(\gamma) = \{ x,y\}$.  We next show that ${\rm fix}(\gamma)\subset \Lambda (\Phi_k)$.  Since $\Gamma^0$ is normal in $\Gamma$ and ${\rm fix}(\gamma)\subset\Lambda(\Gamma) = {\mathbb S}^1_\infty$, we see that ${\rm fix}(\gamma)\subset\Lambda(\Gamma^0)= {\mathbb S}^1_\infty$.  To show that $x\in\Lambda(\Phi_k)$, let $\{ x_n\}$ be a sequence of distinct points of $\Lambda(\Gamma^0)$ converging to $x$ and lying on the same side of $\ell$ as $X_k$; this makes sense, as the endpoints at infinity of $\ell$ separate $\Lambda(\Gamma^0) = {\mathbb S}^1_\infty$.  For each $n$, there exists $\varphi_n\in \Phi_k$ so that the translate $\varphi_n (\ell)$ of $\ell$ lies in the boundary of $X_k$ and separates $X_k$ from $x_n$.  Moreover, since the endpoints at infinity of translates of $\ell$ are disjoint, we can pass to a subsequence so that the $\varphi_n$ are distinct.  Since the $\varphi_n (\ell)$ converge to $x$, we have that $x\in\Lambda(\Phi_k)$.  (The same argument holds for $y$.)

From the construction, it is evident that $\Lambda(\Phi_0)$ and $\Lambda(\Phi_1)$ lie in the different closed arcs in ${\mathbb S}^1_\infty$ determined by ${\rm fix}(\gamma) = \{ x, y\}$.  Hence, we have that ${\rm fix}(\gamma) =\Lambda(\Phi_1)\cap\Lambda(\Phi_2)$, even though we have that $\Phi_1\cap\Phi_2$ is trivial.

In the latter case, a similar argument shows that $\{ x,y\} = \Lambda(\Phi_0)\cap \Lambda(\varphi_0 \Phi_0 \varphi_0^{-1}) =  \Lambda(\Phi_0)\cap \Lambda(\Phi_1)$.  
{\hfill ${\bf QED}$ \\ \vspace{.15cm}}
\end{proof}

An immediate consequence of this construction is that the points of ${\rm fix}(\gamma)$ cannot be points of $\Lambda_c (\Phi_0)$: if on the contrary there were a point $x \in {\rm fix}(\gamma)\cap\Lambda_c (\Phi_0)$, say, then the discreteness of $\langle \Phi_0,\gamma\rangle\subset \Gamma$ and Corollary \ref{cyclic corollary} would together immediately imply that $\gamma^n\in \Phi_0$ for some $n\ge 1$, which is contrary to the construction. 

We further note that the same proof works in all dimensions, though finding groups that satisfy the initial hypotheses becomes increasingly more difficult as the dimension increases.  At a minimum, we would need a Kleinian group $\Gamma$ acting on ${\mathbb H}^n$ so that ${\mathbb H}^n/\Gamma$ contains a closed embedded codimension-$1$ totally geodesic submanifold $S$ with $\pi_1(S) =\Phi \subset \Gamma$ for which there exists an infinite index non-trivial normal subgroup $\Gamma^0\subset \Gamma$ for which $\Gamma^0\cap\Phi$ is trivial. 

\medskip

In our second example, we make use of the following Lemma, whose proof is an easy exercise in hyperbolic geometry, which appears as Lemma 2.1 of Bishop \cite{bishop bm theorem}.  

\begin{lemma} Given $\theta >0$, there exist constants $C = C(\theta)$, $M = M(\theta)>0$ so that the following is true.  Suppose $\gamma$ is a piecewise geodesic path from $a\in {\mathbb H}^2$ to $b\in {\mathbb H}^2$, that is, $\gamma =\cup_{j=1}^n \gamma_j\subset {\mathbb H}^2$ is a union of disjoint (except for endpoints) geodesics arcs, each of hyperbolic length at least $M$ and such that $\gamma_j$ and $\gamma_{j+1}$ meet at an angle $\ge \theta$.  Then $\gamma$ is within hyperbolic distance $C$ of the geodesic arc connecting $a$ and $b$.  In particular, a (one sided) infinite path with this property is within hyperbolic distance $C$ of an infinite geodesic, hence terminates at a single point of ${\mathbb S}^2_\infty$ and approaches that point inside some non-tangential cone.
\label{piecewise quasigeodesic}
\end{lemma}

\begin{proposition} There exists a non-elementary purely loxodromic Fuchsian group $\Gamma$ and non-trivial subgroups $\Phi$ and $\Theta$ of $\Gamma$ so that $\Lambda(\Phi\cap\Theta) \setminus \Lambda_c(\Phi)\cap\Lambda_c(\Theta)$ is non-empty.
\label{grid example}
\end{proposition}

\begin{proof} Let $S_0$ be a closed orientable surface of genus two.  Let $c$ be the figure-eight curve on $S_0$ going once around each handle of $S_0$, and let $z_0$ be the self-intersection point of $c$.   Label the two simple closed non-separating curves determined by $c$ as $a$ and $b$, and note that these are not geodesics as they both have a corner at $z_0$.  Without loss of generality, we may assume that $S_0$ is equipped with a hyperbolic metric so that the angle of intersection of $c$ at $z_0$ is $\frac{\pi}{2}$ and both $a$ and $b$ have lengths significantly greater than the constant $M = M(\frac{\pi}{2})$ coming from Lemma \ref{piecewise quasigeodesic}.   Let $\alpha$ and $\beta$ be the simple closed non-separating geodesics in the free homotopy classes of $a$ and $b$, respectively.  Let $\Gamma_0$ be a Fuchsian group so that $S_0 = {\mathbb H}^2/\Gamma_0$.

Let $S$ be the ${\mathbb Z}\oplus {\mathbb Z}$ cover of $S_0$ corresponding to $\alpha$ and $\beta$, so that $S$ is an infinite genus surface on which ${\mathbb Z}\oplus {\mathbb Z}$ acts freely with fundamental domain homeomorphic to $S_0 \setminus (\alpha\cup \beta)$ and with quotient $S/{\mathbb Z}\oplus {\mathbb Z} = S_0$.  Let $\pi: S\rightarrow S_0$ be the covering map.   We can visualize $S$ as a regular neighborhood in ${\mathbb R}^3$ of the standard square lattice $L$ in the $xy$-plane.

Fix a point $w_0$ in $\pi^{-1} (z_0)$, which we will use as a basepoint.  Let $A$ be a lift of $a$ to $S$ starting at $w_0$ and $B$ a lift of $b$ to $S$ starting at $w_0$. (If we chose an orientation on $c$, this induces an orientation on both $a$ and $b$ and this would then allow us to choose unique lifts $A$ and $B$ of $a$ and $b$, respectively.)   Let $X_A$ be the translation acting on $S$ and taking $w_0$ to the other endpoint of $A$, and let $X_B$ be the translation acting on $S$ and taking $w_0$ to the other endpoint of $B$.   Note that we can explicitly identify the ${\mathbb Z}\oplus {\mathbb Z}$ group of translations acting on $S$ with $G = \langle X_A, X_B\rangle$, so that $S/G = S_0$.   

There are two intermediate surfaces of interest, namely $S_A = S/\langle X_A\rangle$ and $S_B = S/\langle X_B\rangle$.   Moreover, there exist (necessarily normal) subgroups $\Xi$, $\Xi_A$ and $\Xi_B$ of $\Gamma_0$ so that $S = {\mathbb H}^2/\Xi$, $S_A = {\mathbb H}^2/\Xi_A$ and $S_B = {\mathbb H}^2/\Xi_B$, respectively.   In particular, we have the inclusions $\Xi\subset \Xi_A\subset \Gamma_0$ and $\Xi\subset \Xi_B\subset \Gamma_0$, as well as the relation $\Xi = \Xi_A\cap \Xi_B$. 

The points in $\pi^{-1} (z_0)$ are the translates of $w_0$ by elements of $G$, so that $\pi^{-1} (z_0) = \{ X_A^n \cdot X_B^m (w_0)\: |\:n, m\in {\mathbb Z} \}$.   Using this labelling, we will construct a piecewise geodesic ray from the lifts $A$ and $B$ of the lobes $a$ and $b$ of $c$ to $S$.  

Let $w$ be a piecewise geodesic ray on $S$ constructed from the translates of $A$ and $B$ that starts from $w_0$ and spirals out.  If we wish to be specific, consider the piecewise geodesic $w$ constructed as follows: 
\begin{itemize}
\item We start by taking $B$, which joins $w_0 \sim (0,0)$ to $X_B (w_0) \sim (0,1)$;
\item to this, we adjoin $X_A^{-1}(A)$, which joins $X_B (w_0) \sim (0,1)$ to $X_A^{-1}\cdot X_B (w_0)\sim (-1,1)$;
\item to this, we adjoin the two translates of $B$ joining $X_A^{-1}\cdot X_B (w_0)\sim (-1,1)$ to $X_A^{-1}\cdot X_B^{-1} (w_0)\sim (-1,-1)$;
\item to this, we adjoin the two translates of $A$ joining $X_A^{-1}\cdot X_B^{-1} (w_0)\sim (-1,-1)$ to $X_A\cdot X_B^{-1} (w_0)\sim (1,-1)$;
\item to this, we adjoin the three translates of $B$ joining $X_A\cdot X_B^{-1} (w_0)\sim (1,-1)$ to $X_A\cdot X_B^2 (w_0)\sim (1,2)$;
\item to this, we adjoin the three translates of $A$ joining $X_A\cdot X_B^2 (w_0)\sim (1,2)$ to $X_A^{-2}\cdot X_B^2 (w_0)\sim (-2,2)$,
\end{itemize}
and we continue on in this fashion.  

Lifting $w$ to ${\mathbb H}^2$ and calling the lift $W$, we may use Lemma \ref{piecewise quasigeodesic} to see that there exists a geodesic ray $r$ and a constant $C$ so that $W$ lies in the $C$-neighborhood of $r$.   Let $x$ be the endpoint at infinity of $W$, or equivalently of $r$.  Since the projection of $W$ (or equivalently, of $r$) to $S$ exits every compact subset of $S$, we have that $x\in \Lambda(\Xi) \setminus \Lambda_c (\Xi)$.   

However, projecting $W$ (or equivalently, projecting $r$) to $S_A$, we see that $W$ exits both ends of $S_A$.   In particular, there exists a compact set $L$ in $S_A$ and a sequence $\{ x_n\}\subset {\mathbb H}^2$ of points in $W$ converging to $x$ (in the Euclidean topology on ${\mathbb H}^2\cup {\mathbb S}^1_\infty$) so that the $x_n$ project to $L$.  This is enough to conclude that $x\in \Lambda_c (\Xi_A)$.  We may use the same argument on $S_B$ without any change and hence conclude that $x\in \Lambda_c (\Xi_B)$ as well.  

Setting $\Phi = \Xi_A$ and $\Theta = \Xi_B$, we have completed the proof of the Proposition.  {\hfill ${\bf QED}$ \\ \vspace{.15cm}}
\end{proof}

\end{document}